\documentclass[11pt]{amsart}
\usepackage{mathrsfs,amsthm,amsmath,amssymb,bbm,pdfsync,prettyref,hyperref,graphicx,fullpage,esint}

\theoremstyle{plain}
\newtheorem{thm}{\protect\theoremname}
\newtheorem{lem}[thm]{\protect\lemmaname}
\newtheorem{prop}[thm]{Proposition} 
\newtheorem{cor}[thm]{Corollary}
\theoremstyle{definition}
\newtheorem{rem}[thm]{Remark}

\newcommand{\cA}{\mathcal{A}}
\newcommand{\cX}{\mathcal{X}}

\newcommand{\cF}{\mathcal{F}}

\newcommand{\cL}{\mathcal{L}}
\newcommand{\cP}{\mathcal{P}}

\newcommand{\R}{\mathbb{R}}

\newcommand{\E}{\mathbb{E}}

\newcommand{\abs}[1]{\lvert#1\rvert}
\newcommand{\norm}[1]{\lvert\lvert#1\rvert\rvert}

\makeatother

  \providecommand{\lemmaname}{Lemma}
\providecommand{\theoremname}{Theorem}
\newrefformat{prop}{Proposition \ref{#1}}
\begin{document}

\title{A Dynamic Programming Approach to the Parisi Functional}

\author{Aukosh Jagannath}
\email{aukosh@cims.nyu.edu}
\author{Ian Tobasco}
\email{tobasco@cims.nyu.edu}
\address{Courant Institute of Mathematical Sciences, 251 Mercer St.\ NY, NY, USA, 10012}

\keywords{Parisi formula; Sherrington-Kirkpatrick model; Dynamic Programming}
\subjclass[2010]{Primary: 60K35, 82B44, 82D30, 49N90; Secondary: 35Q82, 35K58, 49S05} 

\begin{abstract}
G.\ Parisi predicted an important variational formula for the thermodynamic limit of the intensive free energy for a class of mean field spin glasses. 
In this paper, we present an elementary approach to the study of the Parisi functional using stochastic dynamic programing and semi-linear PDE. 
We give a derivation of important properties of the Parisi PDE avoiding the use of Ruelle Probability Cascades and Cole-Hopf transformations.
As an application, we give a simple proof of the strict convexity of the Parisi functional, which was recently proved by Auffinger and Chen in
\cite{AuffChen14}.
\end{abstract}
\maketitle

\section{Introduction}

Consider the mixed $p$-spin glass model on the hypercube $\Sigma_N=\{-1,1\}^N$, which is given by the Hamiltonian
\[
 H_N(\sigma)= H_N'(\sigma) + h\sum_i \sigma_i
\]
where $H_N'$ is the centered gaussian process on $\Sigma_N$ with covariance 
\[
\E H'_N(\sigma^1)H'_N(\sigma^2) =N\xi((\sigma^1,\sigma^2)/N).
\]
The parameter $\xi$ satisfies $\xi(t)=\sum_{p\geq 1}\beta_{p}^{2}t^{p}$
where we assume there is a positive $\epsilon$ such that $\xi(1+\epsilon)<\infty$, and $h$ is a non-negative real number.
It was predicted by Parisi \cite{MPV87}, and later proved rigorously by Talagrand \cite{TalPF}, and Panchenko \cite{PanchPF14}, that the thermodynamic
limit of the intensive free energy is given by 
\[
\lim_{N\rightarrow\infty}\frac{1}{N}\log\sum_{\sigma\in\Sigma_N} e^{H_N(\sigma)} = \inf_{\mu\in\Pr[0,1]} \cP(\mu;\xi,h) \quad\text{a.s.}
\]
Here $\Pr([0,1])$ is the space of probability measures on $[0,1]$, and the Parisi functional, $\cP$, is given by
\[
\cP(\mu;\xi,h)=u_{\mu}(0,h)-\frac{1}{2}\int_{0}^{1}\xi^{\prime\prime}(t)\mu[0,t]t\,dt,
\]
where $u_{\mu}$ solves the Parisi PDE:
\[
\begin{cases}
\partial_{t}u_{\mu}(t,x)+\frac{\xi^{''}\left(t\right)}{2}\left(\partial_{xx}u_{\mu}(t,x)+\mu\left[0,t\right]\left(\partial_{x}u_{\mu}(t,x)\right)^{2}\right)=0 & (t,x)\in(0,1)\times\R\\
u_{\mu}(1,x)=\log\cosh(x).
\end{cases}
\]
In the case that $\mu$ has finitely many atoms, the existence of a solution of the Parisi PDE and its regularity properties 
are commonly proved using the Cole-Hopf transformation
and Ruelle Probability Cascades. A continuity argument is then used to extend the definition of $u_\mu$ to general $\mu$ and 
to prove corresponding regularity properties. Such approaches 
do not address the question of uniqueness of solutions.
See \cite{TalBK11Vol2,PanchSKBook,AuffChen13,AuffChen14} for a summary of these results.
 
In this note, we present a different approach. In \prettyref{sec:wellposedness}, we prove the existence, uniquness, and regularity of the Parisi PDE using 
standard arguments from semi-linear parabolic PDEs.
\begin{thm}\label{thm:mainthm}
The Parisi PDE admits a unique weak solution which is continuous,
differentiable in time at continuity points of $\mu$, and smooth in space. 
\end{thm}
\noindent See \prettyref{sec:wellposedness} for the precise statement of this result,
and in particular for the definition of weak solution.
Due to the non-linearity of the Parisi PDE, low regularity of the coefficients, loss of uniform ellipticity at $t=0$, 
and unboundedness of the initial data, the proof of \prettyref{thm:mainthm}
requires the careful application of many different (though relatively standard) arguments in tandem.

The presentation of a PDE driven approach to the study of this functional is not only of interest to experts 
in the field of spin glasses, but may also be of interest to practitioners of 
the Calculus of Variations, PDEs, and Stochastic Optimal Control. 
There are many important, purely analytical questions surrounding this functional that 
must be addressed before further progress on questions in spin glasses can be made. 
See \cite{TalBK11Vol2,TalBK11,TalPM06} for a discussion.
Some of these questions are thought to be intractable to the methods currently used in 
the spin glass literature but appear to be well-suited to the techniques of the aforementioned fields; 
as such it is important to present the study of this functional in a language that is both basic and palatable to their practitioners. 

Besides its intrinsic interest, the preceding theorem has useful applications to the study of the Parisi functional. 
After proving the existence of a sufficiently regular solution to the above PDE,
we can use elementary arguments from stochastic analysis to prove many
important and basic properties of this functional, such as fine estimates on the solution of the Parisi PDE and the strict
convexity of the Parisi functional itself.

As a first application of this type, we further develop the well-posedness theory of the Parisi PDE by quantitatively proving the continuity of the solution in the measure $\mu$.
We also prove sharp bounds on some of the derivatives of the solution. Such bounds
are important to the proofs of many important results regarding the Parisi functional, see for example Talagrand's proof of the Parisi formula in \cite{TalBK11Vol2} and also \cite{AuffChen13,AuffChen14,TalPM06}. They were 
previously proved using manipulations of the Cole-Hopf transformation and Ruelle Probability Cascades \cite{TalBK11Vol2}. 
This is presented in \prettyref{sec:cts-dep}.

As a further demonstration how \prettyref{thm:mainthm} can be combined with methods from stochastic optimal control,
we present a simple proof of the strict convexity of the Parisi functional.
As background, recall the prediction by Parisi \cite{MPV87} that the minimizer of the Parisi functional should be unique and
should serve the role of the order parameter in these systems. 
The question of the strict convexity of $\cP$ was first posed by Panchenko in \cite{PanchConv05} 
as a way to prove this uniqueness. It was studied by Panchenko \cite{PanchConv05}, Talagrand \cite{TalPM06,TalPF}, Bovier and Klimovsky \cite{BovKlim08},
and Chen \cite{Chen13}, and finally resolved by Auffinger and Chen in their fundamental work \cite{AuffChen14}. 
The work of Auffinger and Chen rested on a variational representation
of the log-moment generating functional of Brownian motion \cite{BD98,Flem78}, which they combine with approximation arguments to give a variational representation for the solution of the Parisi PDE.
We note here that an early version of this variational representation appeared in \cite{BovKlim08},
where it is shown, using the theory of viscosity solutions, to hold when the coefficient $\mu[0,t]$ is piecewise continuous with finitely many jumps.

Since the Parisi PDE is a Hamilton-Jacobi-Bellman equation, it is natural to obtain the desired variational representation for its solution
as an application of the dynamic programming principle from stochastic optimal control theory.
The required arguments are elementary, and are commonly used in studying nonlinear parabolic PDEs of the type seen above. 
We prove the variational representation in \prettyref{sec:rep-form}, 
and then deduce from it the strict convexity of the Parisi functional in \prettyref{sec:strict-conv}.

\begin{thm}
The functional $\cP(\mu;\xi,h)$ is strictly convex for all
choices of $\xi$ and $h$. 
\end{thm}

The variational representation which was discussed above is given in \prettyref{lem:rep-pde}. 
From this it follows immediately that one has the following representation for the Parisi Formula.
 
\begin{prop}
The Parisi Formula has the representation
\begin{align*}
\lim_{N\rightarrow\infty}\frac{1}{N}\log\sum_{\sigma\in\Sigma_N} e^{H_N(\sigma)} 
= \inf_{\mu\in\Pr([0,1])}\sup_{\alpha\in\cA_{0}}\E&\left[
\log\cosh\left(\int_{0}^{1}\xi''(s)\mu[0,s]\alpha_sds+\int_0^t\sqrt{\xi''(s)}dW_s+h\right)\right.\\
&\left.-\frac{1}{2}\int_{0}^{1}\xi''(s)\mu[0,s]\left(\alpha^{2}_s+s\right)ds\right]
\end{align*}
where $\cA_0$ consists of all bounded processes on $[0,1]$ that are progressively measurable with respect to the filtration of Brownian motion.
\end{prop}


\section*{Acknowledgements}
We would like to thank Antonio Auffinger for encouraging the preparation of this paper, and asking A.J.\ if one could prove strict convexity using dynamic programming techniques. We would also like to thank Anton Klimovsky for bringing our attention
to \cite{BovKlim08}. Finally, we would like to thank 
our advisors G.\ Ben Arous and R.V.\ Kohn for their support.
This research was conducted while A.J.\ was supported by an NSF Graduate Research Fellowship DGE-0813964, 
NSF Grant DMS-1209165, and NSF grant OISE-0730136, and while I.T.\ was supported by an NSF Graduate Research Fellowship
DGE-0813964, NSF Grant OISE-0967140, and NSF Grant DMS-1311833. 

\section{Well-posedness of the Parisi PDE} \label{sec:wellposedness}


Let $u:\left[0,1\right]\times\R\to\R$ be a continuous function with
essentially bounded weak derivative $\partial_{x}u$. We call $u$
a \emph{weak solution} of the Parisi PDE if it satisfies 

\[
0=\int_{0}^{1}\int_{\R}-u\partial_{t}\phi+\frac{\xi''\left(t\right)}{2}\left(u\partial_{xx}\phi+\mu\left[0,t\right]\left(\partial_{x}u\right)^{2}\phi\right)\, dxdt+\int_{\R}\phi\left(1,x\right)\log\cosh x\, dx
\]
for every $\phi\in C_{c}^{\infty}\left((0,1]\times\R\right).$  
We now state the precise version of \prettyref{thm:mainthm} from the introduction.
\begin{thm}\label{thm:PDE-existence-reg}
There exists a unique weak solution $u$ to the Parisi PDE. The solution
$u$ has higher regularity:
\begin{itemize}
\item $\partial_{x}^{j}u\in C_{b}\left(\left[0,1\right]\times\R\right)$
for $j\geq1$
\item $\partial_{t}\partial_{x}^{j}u\in L^{\infty}\left(\left[0,1\right]\times\R\right)$
for $j\geq0$. 
\end{itemize}
For all $j\geq1$, the derivative $\partial_{x}^{j}u$ is a weak solution
to
\[
\begin{cases}
\partial_{t}\partial_{x}^{j}u+\frac{\xi''\left(t\right)}{2}\left(\partial_{xx}\partial_{x}^{j}u+\mu\left[0,t\right]\partial_{x}^{j}\left(\partial_{x}u\right)^{2}\right)=0 & \left(t,x\right)\in\left(0,1\right)\times\R\\
\partial_{x}^{j}u\left(1,x\right)=\frac{d^{j}}{dx^{j}}\log\cosh x & x\in\R
\end{cases}.
\]

\end{thm}
\begin{rem}
The solution described in \cite{AuffChen13} can be shown to be a weak solution of the Parisi PDE, 
 using the approximation methods developed there. It was also shown in \cite{AuffChen13} that this solution has the higher
 regularity described above. 
 \end{rem}
 \begin{rem}
  The reader may notice that the essential boundedness of $\partial_xu$ is not strictly necessary 
  to make sense of the definition of weak solutions. It is used in the proof of uniqueness in an essential way, however 
  we do not claim that this proof is optimal by any means. 
   \end{rem}

\noindent Continous dependence is proved in \prettyref{sec:cts-dep}.

We begin the proof of \prettyref{thm:PDE-existence-reg}. After performing the time change $t\to s\left(t\right)=\frac{1}{2}\left(\xi'\left(1\right)-\xi'\left(t\right)\right)$
and extending the time-changed CDF $\mu\left[0,s^{-1}\left(t\right)\right]$
by zero, we are led to consider the semi-linear parabolic PDE 
\begin{equation}
\begin{cases}
\partial_{t}u-\Delta u=m\left(t\right)u_{x}^{2} & \left(t,x\right)\in\R_{+}\times\R\\
u\left(0,x\right)=g\left(x\right) & x\in\R
\end{cases}\label{eq:semilinearPDE}
\end{equation}
where $g\left(x\right)=\log\cosh x$ and $m\left(t\right)=\mu\left[0,s^{-1}\left(t\right)\right]1_{t\leq(\xi'(1)-\xi'(0))/2}$.
We carry over the definition of weak solution from before: a continuous function $u:[0,\infty)\times\R\to \R$ with essentially
bounded weak derivative $\partial_x u$ is a \emph{weak solution} to \prettyref{eq:semilinearPDE} if it satisfies
\[
0=\int_{0}^{\infty}\int_{\R} u\partial_{t}\phi+
u\partial_{xx}\phi+m(t)\left(\partial_{x}u\right)^{2}\phi\, dxdt+\int_{\R}\phi\left(0,x\right)g(x)\, dx
\]
for every $\phi\in C_{c}^{\infty}\left([0,\infty)\times\R\right).$ 
Evidently, the existence, uniqueness, and regularity theory of weak solutions to the Parisi PDE is 
captured by that of \prettyref{eq:semilinearPDE}.

Our proof of the well-posedness of \prettyref{eq:semilinearPDE} boils down to the
study of a certain fixed point equation, which we introduce now.
Let $e^{t\Delta}$ be the heat semigroup on $\R$, i.e.
\[
\left(e^{t\Delta}h\right)\left(x\right)=\frac{1}{\sqrt{4\pi t}}\int_{\R}e^{-\frac{\abs{x-y}^{2}}{4t}}h\left(y\right)\, dy.
\]
Then, $u$ weakly solves \prettyref{eq:semilinearPDE} if and only
if $u$ satisfies
\begin{equation}
u\left(t\right)=e^{t\Delta}g +\int_{0}^{t}e^{\left(t-s\right)\Delta}m\left(s\right)u_{x}^{2}
\left(s\right)\, ds.\label{eq:FP}
\end{equation}
This is an application of Duhamel's principle (see e.g. \cite[Ch. 2]{EvansPDE}). For completeness, we present this in \prettyref{prop:duhamel}.

In Sections \ref{sec:FPexists}-\ref{sec:FPisunique} below, we
prove the existence, uniqueness, and regularity of fixed points of \prettyref{eq:FP}
on a certain complete metric space.
The properties of
$g$ and $m$ we will be using are that
\begin{itemize}
\item $g'\in L^{\infty}$ and $\frac{d^{j}}{dx^{j}}g\in L^{2}\cap L^{\infty}$
for $j\geq2$
\item $m$ is a monotonic function of time alone and $\norm{m}_{\infty}\leq1$.
\end{itemize}
These properties will inform our choice of space
on which to study \prettyref{eq:FP}. The exact bound on $m$ does not matter,
but we include it for convenience.

Once \prettyref{thm:PDE-existence-reg} is established, one can give a quick proof of the final component of 
wellposedness, namely the continuity of the map from
$\mu$ to the corresponding solution of the Parisi PDE, using standard SDE techniques. This is in \prettyref{sec:cts-dep}.

The notation $\lesssim_{c}$ denotes an inequality that is true up to a universal constant that depends only on $c$.
Throughout the proofs below, we will use two elementary estimates for the heat kernel which we record here:
\begin{equation} \label{eq:heat_ests}
 \norm{e^{t\Delta}}_{L^p\to L^p} \leq 1 \quad \text{and} \quad \norm{\partial_x e^{t\Delta}}_{L^p\to L^p} \lesssim \frac{1}{\sqrt{t}}.
\end{equation}

\subsection{Existence of a fixed point} \label{sec:FPexists}

We prove existence of a fixed point to \prettyref{eq:FP}. First we
show there exists a solution for short-times $t<T_{*}$, then by using
an a priori estimate we prove a solution exists for all time.

Short-time existence comes via a contraction mapping argument. Define
the Banach space
\[
\cX=\left\{ \psi\in L^{\infty}\left(\R\right)\right\} \cap\left\{ \psi_{x}\in L^{\infty}\left(\R\right)\right\} \cap\left\{ \psi_{xx}\in L^{2}\left(\R\right)\right\} 
\]
with the norm 
\[
\norm{\psi}_{\cX}=\norm{\psi}_{\infty}\vee\norm{\psi_{x}}_{\infty}\vee\norm{\psi_{xx}}_{2},
\]
and for each $T>0$ define the complete metric space
\[
X_{T}^{h}=\left\{ e^{t\Delta}h+\phi\,:\,\phi\in L^{\infty}\left(\left[0,T\right];\cX\right)\right\} \cap\left\{ \norm{\phi_{x}}_{L^{\infty}\left(\left[0,T\right]\times\R\right)}\leq\norm{h'}_{\infty},\,\norm{\phi_{xx}}_{L^{\infty}\left(\left[0,T\right];L^{2}\left(\R\right)\right)}\leq\norm{h''}_{2}\right\} 
\]
with the distance
\[
d_{X_{T}^{h}}\left(u,v\right)=\norm{u-v}_{L^{\infty}\left(\left[0,T\right];\cX\right)}.
\]
The symbol $h$ in the definition of the space refers to the initial
data, which is assumed to satisfy $h'\in L^\infty$ and $h''\in L^2$.

Given $u\in X_{T}^{h}$ define the map
\begin{equation}
A\left[u\right]=e^{t\Delta}h+\int_{0}^{t}e^{\left(t-s\right)\Delta}m\left(s\right)u_{x}^{2}\left(s\right)\, ds.\label{eq:mapA}
\end{equation}

\begin{lem}
\label{lem:(short-time-existence)}(short-time existence) Let
\begin{equation}
T_{*}\left(h\right)=\min\left\{ 1,\left[C\cdot\left(\norm{h'}_{\infty}+\norm{h''}_{2}\right)\right]^{-2}\right\} \label{eq:theshorttime}
\end{equation}
where $C\in\R_{+}$ is a universal constant. Then for all $T\in(0,T_{*})$, 
\begin{itemize}
\item (self-map) $A:X_{T}^{h}\to X_{T}^{h}$ 
\item (strict contraction) There exists $\alpha<1$ such that 
\[
d_{X_{T}^{h}}\left(A\left[u\right],A\left[v\right]\right)\leq\alpha\cdot d_{X_{T}^{h}}\left(u,v\right),\quad u,v\in X_{T}^{h}.
\]

\end{itemize}
Therefore for every $T<T_{*}\left(h\right)$ there exists $u\in X_{T}^{h}$
satisfying $u=A\left[u\right]$.\end{lem}
\begin{proof}
First we prove $A$ is a self-map. Let $u\in X_{T}^{h}$ and call
\[
\psi=A\left[u\right]-e^{t\Delta}h=\int_{0}^{t}e^{\left(t-s\right)\Delta}m\left(s\right)u_{x}^{2}\left(s\right)\, ds.
\]
Note that
\begin{align*}
\psi_x &= \int_0^t \partial_x e^{(t-s)\Delta} mu_x^2(s)\,ds \\
\psi_{xx} &= \int_0^t \partial_x e^{(t-s)\Delta} 2mu_xu_{xx}(s)\,ds.
\end{align*}
The estimates in \prettyref{eq:heat_ests} and the definition of $X_{T}^{h}$
imply the bounds 
\begin{align*}
\norm{\psi}_{L^{\infty}\left(\left[0,T\right]\times\R\right)} & \lesssim T\norm{h'}_{\infty}^{2}\\
\norm{\psi_{x}}_{L^{\infty}\left(\left[0,T\right]\times\R\right)} & \lesssim T^{1/2}\norm{h'}_{\infty}^{2}\\
\norm{\psi_{xx}}_{L^{\infty}\left(\left[0,T\right];L^{2}\left(\R\right)\right)} & \lesssim T^{1/2}\norm{h'}_{\infty}\norm{h''}_{2}.
\end{align*}
Therefore there is a universal constant $C\in\R_{+}$ such that $A:X_{T}^{h}\to X_{T}^{h}$
whenever 
\[
T\leq T_{0}\left(h\right)=\left(C\norm{h'}_{\infty}\right)^{-2}.
\]

Now we prove $A$ is a strict contraction. Let $u,v\in X_{T}^{h}$ and
call
\[
D=A\left[u\right]-A\left[v\right]=\int_{0}^{t}e^{\left(t-s\right)\Delta}m\left(s\right)\left(u_{x}^{2}\left(s\right)-v_{x}^{2}\left(s\right)\right)\, ds.
\]
The estimates in \prettyref{eq:heat_ests} and the definition of $X_{T}^{h}$
give
\[
d_{X_{T}^{h}}\left(A\left[u\right],A\left[v\right]\right)\leq C\cdot\max\left\{ T\norm{h'}_{\infty},T^{1/2}\norm{h'}_{\infty},T^{1/2}\left(\norm{h''}_{2}+\norm{h'}_{\infty}\right)\right\} d_{X_{T}^{h}}\left(u,v\right)
\]
where $C\in\R_{+}$ is a universal constant. Therefore, if 
\[
T_{1}\left(h\right)=\min\left\{ 1,\left[C\cdot\left(\norm{h'}_{\infty}+\norm{h''}_{2}\right)\right]^{-2}\right\} 
\]
then $A$ is a strict contraction on $X_{T}^{h}$ for all $T<T_{0}\wedge T_{1}$.
Since $T_{1}\leq T_{0}$ we may take $T_{*}=T_{1}$.
\end{proof}

To prove the existence of a global-in-time solution to \prettyref{eq:FP} we will work in the space
\[
X_{T}=\left\{ e^{t\Delta}g+\phi\,:\,\phi\in L^{\infty}\left(\left[0,T\right];\cX\right)\right\} 
\]
defined for each $T\in\R_+$. Note $X_T^g\subset X_T$ so that by \prettyref{lem:(short-time-existence)}, if we take $T<T_*(g)$
then there exists $u\in X_{T}$ satisfying the fixed point equation \prettyref{eq:FP}. To extend $u$ to all of time we require the following a priori estimates.

\begin{lem}
\label{lem:(a-priori-estimates)}(a priori estimates) Let $T\in\R_+$ and assume $u\in X_{T}$ satisfies \prettyref{eq:FP}. Then
\begin{align*}
\norm{u_{x}}_{L^{\infty}\left(\left[0,T\right]\times\R\right)} & \leq\norm{g'}_{\infty}\\
\norm{u_{xx}}_{L^{\infty}\left(\left[0,T\right];L^{2}\left(\R\right)\right)} & \leq\norm{g''}_{2}\exp\left(\norm{g'}_{\infty}^{2}T\right).
\end{align*}
\end{lem}
\begin{proof}
The estimate on $u_{x}$ is derived by the maximum principle. By Corollary
\ref{cor:ptwisePDE} (see below) we have
\[
\partial_{t}^{\pm}u_{x}\left(t,x\right)-\Delta u_{x}\left(t,x\right)=2m\left(t\pm\right)u_{x}\partial_{x}u_{x}\left(t,x\right),\quad\forall\,\left(t,x\right)\in\left(0,T\right)\times\R
\]
and by assumption $u_x$ is bounded. Now the usual proof of the maximum principle for linear parabolic PDE in unbounded domains goes through \cite{friedman2013partial}.

For the estimate on $u_{xx}$ observe that 
\[
u_{x}=e^{t\Delta}g'+\int_{0}^{t}e^{\left(t-s\right)\Delta}2mu_{x}u_{xx}\left(s\right)\, ds,
\]
so by a standard energy estimate (see \prettyref{lem:energyestimate} below) we have for almost every $t\leq T$
\[
\norm{u_{xx}}_{L^{2}\left(\R\right)}^{2}(t)\leq2\norm{u_{x}}_{L^{\infty}\left(\left[0,T\right]\times\R\right)}^{2}\int_0^t\norm{u_{xx}}_{L^{2}\left(\R\right)}^{2}(s)ds+\norm{g''}_{2}^{2}.
\]
The desired result follows from Gronwall's inequality \cite{EvansPDE} and the
a priori bound on $u_{x}$.\end{proof}
\begin{cor}
(global existence) For each $T\in\R_+$, there exists $u_T\in X_T$ satisfying \prettyref{eq:FP}. 
The solutions $\{u_T\}_{T\in\R_+}$ so produced agree on their common domains.
\end{cor}
\begin{proof}

Define the maximal time of existence $T_M$ to be the supremum over $T\in\R_+$ such that
there exists $u_T\in X_T$ satisfying \prettyref{eq:FP}. If $T_M<\infty$
then by \prettyref{lem:(short-time-existence)} we must have 
\[
\limsup_{T\uparrow T_M}\ \norm{(u_T)_x}_{L^\infty([0,T]\times\R)}+\norm{(u_T)_{xx}}_{L^\infty([0,T];L^2(\R))} = \infty,
\]
otherwise we could construct a solution extending for times beyond $T_M$.
Therefore by \prettyref{lem:(a-priori-estimates)} we must have $T_M=\infty$. 

A quick application of \prettyref{lem:uniqueness} shows that $u_T=u_{T'}$ for $t\leq T\wedge T'$.

\end{proof}

\subsection{Regularity of fixed points}

One proves the higher regularity of the fixed point $u$ by a parabolic bootstrapping procedure.
\begin{lem}
\label{lem:higher-regularity}(higher regularity) Assume $u\in X_{T}$ satisfies \prettyref{eq:FP}. Then $u$ satisfies
\begin{itemize}
\item $\partial_{x}^{j}u\in L^{\infty}\left(\left[0,T\right];L^{2}\left(\R\right)\cap L^{\infty}\left(\R\right)\right)$
for $j\geq2$
\item $\partial_{t}u\in L^{\infty}\left(\left[0,T\right]\times\R\right)$
and $\partial_{t}\partial_{x}^{j}u\in L^{\infty}\left(\left[0,T\right];L^{2}\left(\R\right)\cap L^{\infty}\left(\R\right)\right)$
for $j\geq1$.
\end{itemize}
\end{lem}
\begin{proof}
Let us describe the first step of the argument. Since $u\in X_T$ we have $u_{x}\in L_{tx}^{\infty}$ and $u_{xx}\in L_{t}^{\infty}L_{x}^{2}$.
Our goal will be to deduce $u_{xx}\in L_{tx}^{\infty}$ and $u_{xxx}\in L_{t}^{\infty}L_{x}^{2}$.
It will be important to note we are working on the finite-time domain
$\left[0,T\right]\times\R$, so that in particular $L_{t}^{\infty}L_{x}^{2}\subset L_{tx}^{2}$.

Start by writing
\[
u_{x}=e^{t\Delta}g'+\int_{0}^{t}e^{\left(t-s\right)\Delta}2mu_{x}u_{xx}\left(s\right)\, ds,
\]
then by \prettyref{lem:energyestimate} we get $u_{xxx}\in L_{tx}^{2}.$
Since $mu_{x}u_{xx}\in L_{t}^{\infty}L_{x}^{2}$, $g''\in L^{\infty}$
and 
\[
u_{xx}=e^{t\Delta}g''+\int_{0}^{t}\partial_{x}e^{\left(t-s\right)\Delta}2mu_{x}u_{xx}\left(s\right)\, ds,
\]
we conclude that $u_{xx}\in L_{tx}^{\infty}$. Here we have used that $\int_0^t\partial_xe^{(t-s)\Delta} \,ds\,:\, L^\infty_tL^2_x\to L^\infty_{tx}$ which follows from \prettyref{eq:heat_ests}.

Now 
\[
\partial_{x}\left(mu_{x}u_{xx}\right)=m\left(u_{xx}^{2}+u_{x}u_{xxx}\right)\in L_{tx}^{2}, 
\text{ so that }
u_{xx}=e^{t\Delta}g''+\int_{0}^{t}e^{\left(t-s\right)\Delta}2m\left(u_{xx}^{2}+u_{x}u_{xxx}\right)\, ds
\]
and finally we conclude $u_{xxx}\in L_{t}^{\infty}L_{x}^{2}$ using
\prettyref{lem:energyestimate} again.

The rest of the estimates on $\partial_{x}^{j}u$ are proved in the
same way; the $\partial_{t}\partial_{x}^{j}u$ estimates follow easily.
\end{proof}
There is a sense in which the weak solution $u$ is a classical solution.
\begin{cor}
\label{cor:ptwisePDE} Let $u\in X_T$ satisfy \prettyref{eq:FP}. Then for all $j\geq0$ we have 
\begin{itemize}
\item $\partial_{x}^{j}u$ exists pointwise and is continuous 
\item the left/right derivatives $\partial_{t}^{\pm}\partial_{x}^{j}u$
exist pointwise, and $\partial_t \partial_{x}^{j} u$ exists at continuity points of $m$
\end{itemize}
Moreover, we have that
\[
\partial_{t}^{\pm}\partial_{x}^{j}u\left(t,x\right)-\Delta\partial_{x}^{j}u\left(t,x\right)=m\left(t\pm\right)\partial_{x}^{j}\left[u_{x}^{2}\right]\left(t,x\right),\quad\forall\,\left(t,x\right)\in\left(0,T\right)\times\R.
\]

\end{cor}

For completeness, we record the energy estimate which was used above. The proof is standard (see \cite{EvansPDE}) and is omitted.
\begin{lem}\label{lem:energyestimate} 
Let $h$ be weakly differentiable with $h'\in L^2$ and let $f\in L^2([0,T]\times \R)$. Then
\[
 \psi(t) = e^{t\Delta}h + \int_0^t e^{(t-s)\Delta}f(s)\,ds
\]
satisfies 
\[
 \norm{\psi_x}^2_{L^\infty([0,T];L^2(\R))} + \norm{\psi_{xx}}^2_{L^2([0,T]\times \R)} \leq \norm{f}^2_{L^2([0,T]\times \R)} + \norm{h'}^2_{L^2(\R)}.
\]

\end{lem}

\subsection{Uniqueness of fixed points} \label{sec:FPisunique}

Since we used a contraction mapping argument to construct fixed points
for \prettyref{eq:mapA} in the spaces $X_{T}^{h}$, we have implicitly
demonstrated a uniqueness theorem there. The following result achieves
uniqueness without mention of the second derivative $u_{xx}$.
\begin{lem}\label{lem:uniqueness}
Assume $u,v:\left[0,T\right]\times\R\to\R$ are weakly differentiable and that $u_{x},v_{x}$ are essentially bounded.
Then if $u,v$ satisfy the fixed point equation \prettyref{eq:FP}, it follows $u=v$.\end{lem}
\begin{proof}
In the following, $C$ denotes a universal constant which may change from line to line.
Let $d=u-v$, then by assumption we have
\[
d\left(t\right)=\int_{0}^{t}e^{\left(t-s\right)\Delta}m\left(s\right)\left(u_{x}+v_{x}\right)d_{x}\left(s\right)\, ds,\quad t\leq T.
\]
Therefore 
\[
 d_x(t) = \int_0^t \partial_x e^{(t-s)\Delta} m(s) (u_x + v_x)d_x(s) \,ds,\quad t\leq T.
\]
Using the second heat kernel estimate in \prettyref{eq:heat_ests}, we conclude the contractive estimate
\[
\norm{d_x}_{L^{\infty}\left([0,t]\times\R\right)}\leq C\norm{u_{x}+v_{x}}_{L^{\infty}\left(\left[0,T\right]\times\R\right)}\int_{0}^{t}\frac{1}{\sqrt{t-s}}\norm{d_{x}\left(s\right)}_{L^{\infty}\left(dx\right)}\, ds
\]
for all $t\leq T$. It now
follows from an iterative argument that $d_{x}=0$, and hence that $d=0$. To see this note that if 
$d_x =0$ on $[0,t_1]\times\R$, then by the contractive estimate above, 
\[
 \norm{d_x}_{L^\infty\left([t_1,t]\times\R\right)} \leq C\norm{u_{x}+v_{x}}_{L^{\infty}\left(\left[0,T\right]\times\R\right)} \sqrt{t-t_1}\norm{d_x}_{L^\infty\left([t_1,t]\times\R\right)}
\]
for all $t\in[t_1,T]$. Therefore $d_x=0$ on $[0,t_1+\epsilon]$ where $\epsilon$ depends only on the $L^\infty$ bounds on $u_x,v_x$. This completes the proof.
\end{proof}
\subsection{Continuous dependence of solutions}\label{sec:cts-dep}
For convenience we metrize
the weak topology on the space of probability measures on the interval
$\Pr\left[0,1\right]$ with the metric
\[
d\left(\mu,\nu\right)=\int_{0}^{1}\abs{\mu\left[0,s\right]-\nu\left[0,s\right]}\, ds.
\]
\begin{lem}
Let $\mu,\nu\in\Pr[0,1]$ and $u,v$ be the corresponding solutions to 
the Parisi PDE. Then
\begin{align*}
 \norm{u-v}_{\infty} & \leq \xi''\left(1\right) d(\mu,\nu) \\
 \norm{u_x-v_x}_{\infty} &\leq \exp\left(\xi'\left(1\right)-\xi'\left(0\right)\right)\xi''\left(1\right) d(\mu,\nu).
\end{align*}
\end{lem}
\begin{rem}
The first inequality is originally due to Guerra \cite{Guer01}.
\end{rem}
\begin{proof}
Let $u$, $v$ solve the Parisi PDE weakly, then $w=u-v$ solves
\[
\begin{cases}
w_{t}+\frac{\xi''}{2}\left(w_{xx}+\mu[0,t]\left(u_{x}+v_{x}\right)w_{x}+\left(\mu[0,t]-\nu[0,t]\right)v_{x}^{2}\right)=0 & \left(t,x\right)\in\left(0,1\right)\times\R\\
w\left(1,x\right)=0 & x\in\R
\end{cases}
\]
weakly. Since $u_{x},v_{x}$ are Lipschitz in space uniformly in time
and bounded in time, we can solve the SDE
\[
dX_{t}=\xi''\left(t\right)\mu\left[0,t\right]\frac{u_{x}+v_{x}}{2}\left(t,X_{t}\right)dt+\sqrt{\xi''\left(t\right)}dW_{t}.
\]
Furthermore, as $w$ weakly solves the above PDE and has the same regularity as $u$ and $v$, we can write 
\[
w\left(t,x\right)=\E_{X_{t}=x}\left(\int_{t}^{1}\frac{1}{2}\xi''\left(s\right)\left(\mu\left[0,s\right]-\nu\left[0,s\right]\right)v_{x}^{2}\left(s,X_{s}\right)\, ds\right)
\]
by \prettyref{prop:low-reg-ito}.
Therefore
\[
\norm{w}_{\infty}\leq\xi''\left(1\right)d\left(\mu,\nu\right)
\]
since $\xi''$ is non-decreasing and $\norm{v_{x}}_{\infty}^{2}\leq1$
by \prettyref{lem:maxprinciples}.

Differentiating the PDE for $w$ in $x$, one finds by similar arguments to \prettyref{prop:low-reg-ito} that $w_x$ has the representation
\[
w_{x}\left(t,x\right)=\E_{X_{t}=x}\left(\int_{t}^{1}E\left(t,s\right)\xi''\left(s\right)\left(\mu\left[0,s\right]-\nu\left[0,s\right]\right)v_{x}v_{xx}\left(s,X_{s}\right)\, ds\right)
\]
where
\[
E\left(t,s\right)=\exp\left(\int_{t}^{s}\xi''\left(\tau\right)\mu\left[0,\tau\right]\frac{v_{xx}+u_{xx}}{2}\left(\tau,X_{\tau}\right)\, d\tau\right).
\]
Using that $\norm{v_{x}}_{\infty}\leq1$ and $\norm{u_{xx}}_{\infty}\vee\norm{v_{xx}}_{\infty}\leq1$
from \prettyref{lem:maxprinciples}, and since $\xi''$ is non-decreasing,
\[
\norm{w_{x}}_{\infty}\leq e^{\xi'\left(1\right)-\xi'\left(0\right)}\xi''\left(1\right)d\left(\nu,\mu\right).
\]
\end{proof}
\begin{lem}
\label{lem:maxprinciples}The solution $u$ to the Parisi PDE satisfies
$\abs{u_{x}}<1$ and $0<u_{xx}\leq1$.\end{lem}
\begin{rem} 
The Auffinger-Chen SDE and the corresponding It\^o's formula's for $u_x$ and $u_{xx}$
used in the proof below were first proved in \cite{AuffChen14} using approximation arguments. 
\end{rem}
\begin{proof}
Using the PDEs for $u_{x},u_{xx}$ given in \prettyref{thm:PDE-existence-reg}, along with \prettyref{prop:low-reg-ito}, we can write
\begin{align*}
u_{x}\left(t,x\right) & =\E_{X_{t}=x}\left(\tanh X_{1}\right)\\
u_{xx}\left(t,x\right) & =\E_{X_{t}=x}\left(\text{sech}^{2}X_{1}+\int_{t}^{1}\xi''\left(s\right)\mu\left[0,s\right]u_{xx}^{2}\left(s,X_{s}\right)\, ds\right)
\end{align*}
where $X_{t}$ solves the Auffinger-Chen SDE
\[
dX_{t}=\xi''\left(t\right)\mu\left[0,t\right]u_{x}\left(t,X_{t}\right)dt+\sqrt{\xi''\left(t\right)}dW_{t}.
\]
The first equality immediately implies the bound on $u_{x}$, and
the second equality implies $u_{xx}>0$. Then by a rearrangement one
finds
\[
u_{xx}\left(t,x\right)=1-\mu[0,t)u_{x}^{2}\left(t,x\right)-\E_{X_{t}=x}\left(\int_{t}^{1}u_{x}^{2}\left(s,X_{s}\right)\, d\mu\left(s\right)\right)
\]
and $u_{xx}\leq1$ follows.
\end{proof}

\section{A variational formulation for the Parisi PDE}\label{sec:rep-form}
In this section we use the methods of dynamic programming (see e.g. \cite{FlemRishOptControl}) to give a new proof of the 
variational formula for the solution of the Parisi PDE.

\begin{lem}\label{lem:rep-pde}
Let $u_{\mu}$ solve the Parisi PDE as above and define the class $\cA_t$ of processes $\alpha_s$ on $[t,1]$ that are
bounded and progressively measurable with respect to Brownian motion. Then 
\begin{equation}\label{eq:PPDE-var-rep}
u_{\mu}(t,x)=\sup_{\alpha\in\cA_{t}}\,\E_{X^\alpha_t=x}\left[-\frac{1}{2}\int_{t}^{1}\xi^{\prime\prime}(s)\mu[0,s]\alpha_{s}^{2}ds+\log\cosh(X_{1}^{\alpha})\right]
\end{equation}
where $X^\alpha_{s}$ solves the  SDE
\begin{equation}
dX^\alpha_{s}=\xi^{\prime\prime}(s)\mu[0,s]\alpha_{s}ds+\sqrt{\xi^{\prime\prime}(s)}dW_{s}
\end{equation}
with initial data $X^\alpha_t=x$. Furthermore, the optimal control satisfies 
\[\mu[0,s]\alpha_s^*=\mu[0,s]u_x(s,X_s)\quad \text{a.s.}\] 
where $X_s$ solves the Auffinger-Chen SDE with the same initial data:
\[
dX_{s}=\xi^{\prime\prime}(s)\mu[0,s]\partial_{x}u(s,X_{s})ds+\sqrt{\xi^{\prime\prime}(s)}dW_s.
\]
\end{lem}
\begin{rem}
This formula was first proved by Auffinger and Chen in \cite{AuffChen14}.
 Taking advantage of the Cole-Hopf representation in the case of atomic $\mu$, 
 they prove the lower bound for every $\alpha$ using Girsanov's lemma and Jensen's inequality. 
 They then verify that their optimal control achieves the supremum, by an application of It\^o's lemma. 
 The uniqueness follows from a convexity argument. In contrast, we recognize the Parisi PDE as a specific 
 Hamilton-Jacobi-Bellman equation. 
 It is well-known that the solution of such an equation can be seen as the value function of a stochastic optimal control problem. 
 As such, this representation can be obtained by a textbook application of ``the verification argument''. 
 This argument simultaneously gives the variational representation and a characterization of the optimizer. We also note
 that the argument presented here is more flexible, as is evidenced by replacing the nonlinearity  $u_x^2$ with $F(u_x)$
in the Parisi PDE, where $F$ is smooth, strictly convex, and has super linear growth. 
In particular, observe that one cannot use the Cole-Hopf transformation on the resulting PDE, but the arguments of this paper follow through \emph{mutatis mutandis}.
\end{rem}
\begin{proof}
Let $u$ solve the Parisi PDE. Notice that the nonlinearity is convex, so if we let
\begin{equation}
\begin{aligned}\label{eq:Landf}
L\left(t,\lambda\right) & =-\xi''\left(t\right)\mu\left[0,t\right]\frac{\lambda^{2}}{2}\\
f(t,\lambda) & =\xi''\left(t\right)\mu\left[0,t\right]\lambda,
\end{aligned}
\end{equation}
then by the Legendre transform we have 
\[
\xi''\left(t\right)\mu\left[0,t\right]\frac{\left(\partial_{x}u\right)^{2}}{2} 
= \xi''(t)\mu[0,t]\sup_{\lambda\in\R} \left\{-\lambda^2/2+\lambda\partial_xu\right\} = \sup_{\lambda\in\R}\left\{ L\left(t,\lambda\right)+f\left(t,\lambda\right)\partial_{x}u\right\}.
\]
Therefore, we can write the Parisi PDE as a Hamilton-Jacobi-Bellman equation:
\[
0=\partial_{t}u+\frac{\xi''\left(t\right)}{2}\partial_{xx}u+\sup_{\lambda\in\R}\left\{ L\left(t,\lambda\right)+f\left(t,\lambda\right)\partial_{x}u\right\}.
\]

Since $\alpha_s$ in $\cA_t$ is bounded and progressively measurable, we can consider the process, $X^\alpha$, which solves the SDE
\[
dX^\alpha=f(s,\alpha_s)ds+\sqrt{\xi''(t)}dW
\]
with initial data $X^\alpha_{t}=x$. This process has corresponding infinitesimal generator
\[
\cL(t,\alpha) = \frac{1}{2}\xi''(t)\partial_{xx}+f(t,\alpha)\partial_x.
\]
Notice that $u$ is a (weak) sub-solution to 
\[
\partial_t u+ \cL(t,\alpha) u+L(t,\alpha)\leq 0
\]
with the regularity obtained in \prettyref{thm:PDE-existence-reg}. It follows from It\^o's lemma (\prettyref{prop:low-reg-ito}) that
\[
u(t,x)\geq\sup_{\alpha\in\cA_t}\,\E_{x}\left[\int_{t}^{1}L\left(s,\alpha_{s}\right)\, ds+\log\cosh\left(X_{1}^{\alpha}\right)\right].
\]
 The result now follows upon observing that the control $u_x(s,X_s)$ achieves equality in the above
 since it achieves equality in the Legendre transform. That this control is in the class $\cA_t$ can be seen
 by an application of the parabolic maximum principle (\prettyref{lem:maxprinciples}). Uniqueness
follows from the fact that $\lambda$ achieves equality in the Legendre transform if and only if 
\[
\xi''(t)\mu[0,t]\lambda = \xi''(t)\mu[0,t]u_x. 
\]

\end{proof}
Applying this representation to the Parisi formula gives Proposition 3.



\section{Strict convexity}\label{sec:strict-conv}
As an application of the above ideas, we give a simple proof of strict convexity of $\cP$.
\begin{thm}
The Parisi Functional is strictly convex.
\end{thm}
\begin{proof}
We will prove $\mu\to u_{\mu}\left(0,h\right)$ is strictly convex. Then 
  \[
  \cP(\mu)=u_{\mu}(0,h)-\frac{1}{2}\int_{0}^{1}\xi^{\prime\prime}(t)\mu[0,t]s\,ds
  \]
will be the sum of a strictly convex and a linear functional, so $\cP$ will be strictly convex.

Recall
\[
u_{\mu}(0,h)=\sup_{\alpha\in\cA_{0}}\,\E_{h}\left[\int_{0}^{1}-\xi''(s)\mu[0,s]\frac{\alpha_s^{2}}{2}\, ds+\log\cosh\left(X_{1}^{\alpha}\right)\right].
\]
 Fix distinct $\mu,\nu\in\Pr\left[0,1\right]$
 and let $\mu_{\theta}=\theta\mu+\left(1-\theta\right)\nu$, $\theta\in\left(0,1\right)$. Let $\alpha^{\theta}$
 be the optimal control for the Parisi PDE associated to $\mu_{\theta}$, so that 
 \[
 u_{\mu_{\theta}}\left(0,h\right)=\E_{h}\left[\int_{0}^{1}-\xi''(s)\mu_{\theta}[0,s]
 \frac{\left(\alpha_s^{\theta}\right)^{2}}{2}\, ds+\log\cosh\left(X_{1}^{\alpha^{\theta}}\right)\right].
 \]

Consider the auxiliary processes $Y_t^{\alpha^{\theta}}$ and $Z_t^{\alpha^{\theta}}$ given by solving
 \[
 dY_ t = \xi''(t)\mu[0,t]\alpha_t^\theta dt+ \sqrt{\xi''(t)}dW_t
 \quad\text{ and }\quad
 dZ_ t = \xi''(t)\nu[0,t]\alpha_t^\theta dt+ \sqrt{\xi''(t)}dW_t
 \]
with initial data $Y_0=Z_0=h$, and note that 
\[
 X_{t}^{\alpha^{\theta}} = \theta Y_t^{\alpha^{\theta}} +(1-\theta)Z_t^{\alpha^{\theta}}.
\]
By the lemma below, $P(Y_1\neq Z_1)>0$. Therefore by the strict convexity of $\log\cosh$ and
the variational representation \prettyref{eq:PPDE-var-rep}, 
\begin{align*}
 u_{\mu_{\theta}}\left(0,h\right)&=
 \E_{h}\left[\int_{0}^{1}-\xi''\left(s\right)\mu_{\theta}\left[0,s\right]\frac{\left(\alpha_s^{\theta}\right)^{2}}{2}\, ds+\log\cosh\left(X_{1}^{\alpha^{\theta}}\right)\right] \\
 &<  \theta\left( \E_{h}\left[\int_{0}^{1}-\xi''\left(s\right)\mu\left[0,s\right]\frac{\left(\alpha_s^{\theta}\right)^{2}}{2}\, ds+\log\cosh\left(Y_{1}^{\alpha^{\theta}}\right)\right]\right) \\
 & + (1-\theta)\left( \E_{h}\left[\int_{0}^{1}-\xi''\left(s\right)\nu\left[0,s\right]\frac{\left(\alpha_s^{\theta}\right)^{2}}{2}\, ds+\log\cosh\left(Z_{1}^{\alpha^{\theta}}\right)\right] \right)\\
 &\leq \theta u_\mu(0,h)+(1-\theta)u_\nu(0,h)
\end{align*}
as desired.
\end{proof}
\begin{lem}
Let $Y_t$ and $Z_t$ be as above. Then $P(Y_1\neq Z_1)>0$.
\end{lem}
\begin{proof}
It suffices to show that
\[
Var(Y_1-Z_1)>0.
\]
By definition we have
\[
 Y_1 - Z_1 = \int_0^1\xi''(s)(\mu[0,s]-\nu[0,s])\alpha_s^\theta \,ds.
\]
Observe that by the PDE for $u_x$ in \prettyref{thm:PDE-existence-reg} and It\^{o}'s lemma (see \prettyref{prop:low-reg-ito}), 
the optimal control $\alpha_t^\theta=(u_{\mu_\theta})_x$ is a martingale,
\[
\alpha^\theta_t - \alpha^\theta_0 = \int_0^t \sqrt{\xi''(s)}u_{xx}(s,X_s)dW_s.
\]
Therefore if we call $\Delta_s=\xi''(s)(\mu[0,s]-\nu[0,s])$, 
\begin{align*}
Var(Y_1-Z_1)&=\E_h\left(\int_0^1\Delta_s(\alpha^\theta_s-\alpha^\theta_0)ds\right)^2=\int_{[0,1]^2} \Delta_s\Delta_t K(s,t)\,dsdt
\end{align*}
where
\[
K(s,t)	=\E_{h}\left[\left(\alpha_{s}^{\theta}-\alpha^\theta_{0}\right)\cdot\left(\alpha_{t}^{\theta}-\alpha^\theta_{0}\right)\right].
\]

Now since $\Delta_s\in L_2[0,1]$, it suffices to show that $K(s,t)$ is positive definite. 
We have
\begin{align*}
K(s,t)	&=\E_{h}\left[\int_{0}^{s}\sqrt{\xi''(s)}u_{xx}\left(s,X_{s}^{\alpha^{\theta}}\right)dW_{s}\cdot\int_{0}^{t}\sqrt{\xi''(t)}u_{xx}\left(t,X_{t}^{\alpha^{\theta}}\right)dW_{t}\right]\\
	&=\int_{0}^{t\wedge s}\xi''(t')\E_{h}u_{xx}^{2}(t',X_{t'}^{\alpha^{\theta}})dt'=p\left(t\wedge s\right)=p(t)\wedge p(s)
\end{align*}
where 
\[
p(s)=\int_{0}^{s}\xi''(t')\E_{h}u_{xx}^{2}(t',X_{t'}^{\alpha^{\theta}})dt'.
\]
By the maximum principle (\prettyref{lem:maxprinciples}),  $u_{xx}>0$, so that $p(t)$ is strictly increasing. 
Since this kernel 
corresponds to a monotonic time change of a Brownian motion, it is positive definite.
\end{proof}

\section{Appendix}
We will say that a function $f:[0,\infty)\times\R \to \R$ with \emph{at most linear growth} if it satisfies an inequality of the form
\[
 \abs{ f(t,x) } \lesssim_T 1+\abs{x} \quad \forall\,T\in\R_+,\ (t,x)\in [0,T]\times\R.
\]
We will say the same in the case that $f:\R\to\R$ with the obvious modifications. In the following we fix a probability space $(\Omega,\cF,P)$ and
let $W_{t}$ be a standard brownian motion with respect to $P.$ Let $\cF_{t}$
be the filtration corresponding to $W_{t}$. 

To make this paper self-contained, we present a version of It\^o's lemma in a lower regularity setting. 
The argument is a modification of \cite[Corr. 4.2.2]{stroock1979multidimensional}.
\begin{prop}\label{prop:low-reg-ito}
Let $a,b:[0,T]\times(\Omega,\cF,P)\rightarrow\R$ be be bounded and progressively
measurable with respect to $\cF_{t}$ and let $a\geq0$. Let $X_{t}$ solve 
\[
dX_{t}=\sqrt{a(t)}dW_{t}+b(t)dt
\]
with initial data $X_{0}=x$. Let $L=\frac{1}{2}a(t,\omega)\Delta+b(t,\omega)\partial_{x}$. 
Finally assume that we have $u$ satisfying:
\begin{enumerate}
\item $u\in C([0,T]\times\R)$ with at most linear growth.
\item $u_{x},$ $u_{xx}\in C_{b}([0,T]\times\R)$
\item $u$ is weakly differentiable in $t$ with essentially bounded weak derivative $u_{t}$,
and which has a representative that is Lipschitz in $x$ uniformly in $t$.
\end{enumerate}
Then $u$ satisfies Ito's lemma: 
\[
u(t,X_{t})-u(s,X_{s})=\int_{s}^{t}\left(\partial_{t}+L\right)u(s',X_{s'})ds'+\int_{s}^{t}u_{x}(s',X_{s'})\sqrt{a(s')}dW_{s'}
\]
\end{prop}
\begin{rem}
 This result is applied throughout the paper to the solution $u$ from \prettyref{thm:mainthm} and its spatial derivatives.
 We note here that, given the regularity in \prettyref{thm:PDE-existence-reg}, 
 the weak derivatives $\partial_t\partial_x^j u$, $j\geq 0$, have representatives satisfying the above Lipschitz property.
\end{rem}

\begin{proof}
To prove this, we will smooth $u$ by a standard mollification-in-time procedure and apply It\^o's lemma. 
Without loss of generality, assume $T=1$
and $s=0$. Extend $u$ to all of space-time by 
\[
u(t,x)=\begin{cases}
u(0,x) & t<0\\
u(1,x) & t>1
\end{cases}.
\]
Abusing notation, we call the extension $u$ and note that it satisfies each of the assumptions above.
Let $\phi(y)\in C_{c}^{\infty}(-1,1)$
with $0\leq\phi\leq 1$ and $\int\phi=1$, and define $\phi_{\epsilon}(s)=\phi(s/\epsilon)/\epsilon$.
Define the time-mollified version of $u$ as
\[
u^{\epsilon}(t,x)=\int_{\R}\phi_{\epsilon}(s)u(t-s,x)ds.
\]

Since $u^{\epsilon}\in C^{1,2}$ has bounded derivatives, and grows at most linearly, Ito's lemma implies that 
\begin{align*}
u^{\epsilon}(t,X_{t})-u^{\epsilon}(0,x) & =\int_{0}^{t}\left(\partial_{t}+L\right)u^{\epsilon}(s,X_{s})ds+\int_{0}^{t}u_{x}^{\epsilon}(s,X_{s})\sqrt{a}dW_{s}\\
 & =\int_{0}^{t}u_{t}^{\epsilon}(s,X_{s})ds+\int_{0}^{t}Lu^{\epsilon}(s,X_{s})ds\int_{0}^{t}u_{x}^{\epsilon}(s,X_{s})\sqrt{a}dW_{s}\\
 & =A_{\epsilon}+B_{\epsilon}+C_{\epsilon}
\end{align*}
for all $\epsilon>0$.
Since these quantities are well-defined at $\epsilon=0$, it suffices
to show their convergence.

First we show the left-hand side converges. Note $u$ is Lipschitz with constant $\norm{\nabla u}_{\infty}$.
Therefore,
\[
\sup_{x\in\R}\sup_{t}\abs{u^{\epsilon}(t,x)-u(t,x)} =\sup_{x\in\R}\sup_{t}\abs{\int\phi(y)\left(u(t-\epsilon y,x)-u(t,x)\right)dy}\leq \norm{\nabla u}_{\infty}\epsilon.
\]
Thus $u^{\epsilon}(t,X_{t})\rightarrow u(t,X_{t})$ uniformly $P$-a.s.

Now we consider the right-hand side. For $A_{\epsilon}$, note that since $u_t$ is Lipschitz in $x$ uniformly in $t$,
by an application of Lebesgue's differentiation theorem, 
we have that $u_{t}^{\epsilon}\rightarrow u_t$
for all $x$, Lebesgue-a.s. in $t$. Thus by the bounded convergence
theorem, we have that 
\[
\sup_{t\in[0,1]}\abs{\int_{0}^{t}u_t^{\epsilon}(s,X_{s})ds-\int_{0}^{t}u_t(s,X_{s})ds}\leq\int_{0}^{1}\abs{u_t^{\epsilon}(s,X_{s})-u_t(s,X_{s})}ds\rightarrow0.
\]
Thus, $A_{\epsilon}\rightarrow A$ uniformly $P$-a.s.

The convergence for $B_{\epsilon}$ follows from a similar argument. Since $u_x,u_{xx}\in$ $C_{b}$,
commuting derivatives with mollification
shows that $u_{x}^{\epsilon}$ and $u_{xx}^{\epsilon}$ converge to
 $u_{x}$ and $u_{xx}$ pointwise. Then, the bounded convergence
theorem 
implies that $B_\epsilon \rightarrow B$ uniformly $P$-a.s. just as before.

Now we prove uniform a.s. convergence of $C_{\epsilon}$ to $C$. Combining the above arguments proves that $C_\epsilon$ 
is uniformly a.s. convergent, so it suffices to check its convergence to $C$ in probability.
By Doob's inequality and Ito's isometry,
\[
P\left(\sup_{t\in[0,1]}\abs{\int_{0}^{t}u_{x}^{\epsilon}\sqrt{a}dW_{s}-\int_{0}^{t}u_{x}\sqrt{a}dW_{s}}\geq\eta\right)
\lesssim_{a}\frac{1}{\eta^{2}}\int_{0}^{1}\E\abs{u_{x}^{\epsilon}-u_{x}}^2\to0
\]
where the last convergence is again by the bounded convergence theorem. \end{proof}
%

We finish with a discussion of Duhamel's principle, which justifies the introduction of the fixed point equation
\prettyref{eq:FP} in the proof of \prettyref{thm:PDE-existence-reg}. Note that since our weak solutions satisfy 
$\partial_x u\in L^\infty$ by definition, they have at most linear growth.

\begin{prop} \label{prop:duhamel}
Suppose that $u,f:[0,\infty)\times\R\to\R$, $g:\R \to \R$ have at most linear growth. Assume that $f$ is Borel measurable,
and that $u$ and $g$ are continuous. Then
\begin{equation}
0=\int_{0}^{\infty}\int_{\R}u\partial_{t}\phi+u\partial_{xx}\phi+f\phi\,dxdt+\int_{\R}\phi\left(0,x\right)g\left(x\right)\,dx\quad\forall\,\phi\in C_{c}^{\infty}\left([0,\infty)\times\R\right)\label{eq:wkeqnforu}
\end{equation}
if and only if 
\begin{equation}
u\left(t\right)=e^{t\Delta}g+\int_{0}^{t}e^{\left(t-s\right)\Delta}f\left(s\right)\,ds\quad\forall\,t\in[0,\infty).\label{eq:fixedptforu}
\end{equation}
\end{prop}
\begin{rem}
  Although the assumption of linear growth is not optimal, it will be sufficient for our application. 
  Implicit here is a uniqueness theorem for weak solutions of the heat equation with at most linear growth. 
  Recall that even classical solutions fail to be unique without certain growth conditions at $\abs{x}=\infty$ (see e.g. \cite[Ch. 7]{JohnPDE}). 
\end{rem}

\begin{proof}
That $u$ satisfies \eqref{eq:wkeqnforu} if
it satisfies \eqref{eq:fixedptforu} is clear
in the case that $f,g$ are smooth and compactly supported. Then,
a cutoff and mollification argument upgrades the result to the given class. 

In the other direction, suppose that $u$ satisfies \eqref{eq:wkeqnforu}. Define the function 
\[
\Theta\left(t,x\right)=u\left(t,x\right)-\left[e^{t\Delta}g\left(\cdot\right)\right]\left(x\right)-\int_{0}^{t}\left[e^{\left(t-s\right)\Delta}f\left(s,\cdot\right)\right]\left(x\right)\,ds,
\]
which is continuous and satisfies $\Theta\left(0,\cdot\right)=0$. By a similar argument
as above, $\Theta$ satisfies the heat equation in the sense of distributions
on $\R_{+}\times\R$. Since the heat operator is hypoelliptic, it follows that $\Theta$ is a classical solution \cite{FollandPDE}. 
By its definition, $\Theta$ grows at most linearly
since the same is true for $u$, $f$, and $g$. By the maximum
principle for the heat equation in unbounded domains \cite{JohnPDE}, we conclude that $\Theta =0$.
\end{proof}
\bibliographystyle{plain}
\bibliography{../../parisimeasures}
\end{document}